\documentclass[11pt, letterpaper]{amsart}
\usepackage{amsxtra}
\usepackage{mathtools}
\usepackage{extarrows}
\usepackage{upgreek}
\usepackage{color}
\usepackage{xcolor}
\usepackage{enumitem}
\usepackage[all]{xy}
\usepackage{subcaption}
\usepackage{rotating}
\setlist[enumerate]{label=\upshape(\arabic*)}
\allowdisplaybreaks[4]
\usepackage[colorlinks=true, citecolor=blue]{hyperref}

\usepackage{amsthm}
\usepackage{amssymb}
\usepackage{amsmath}
\numberwithin{equation}{section}
\usepackage{upref}
\usepackage{tikz}
\usetikzlibrary{calc}

\newtheorem{theorem}{Theorem}[section]
\newtheorem{theorem-definition}[theorem]{Theorem-Definition}
\newtheorem{lemma}[theorem]{Lemma}
\newtheorem{proposition}[theorem]{Proposition}
\newtheorem{corollary}[theorem]{Corollary}

\newtheorem{question}[theorem]{Question}

\newtheorem{eg}[theorem]{Example}
\newtheorem{remark}[theorem]{Remark}

\DeclareMathOperator{\Lie}{Lie}

\DeclareMathOperator{\Pet}{Pet}
\DeclareMathOperator{\rank}{rank}

\begin{document}

\title[Structure constants of Peterson Schubert calculus
]{Structure constants of Peterson Schubert calculus
}



\author{Tao Gui}
\address{(Tao Gui) \newline \indent Beijing International Center for Mathematical Research, Peking University, No.\ 5 Yiheyuan Road, Haidian District, Beijing 100871, P.R. China}
\email{guitao18(at)mails(dot)ucas(dot)ac(dot)cn}

\author{Yuqi Jia}
\address{(Yuqi Jia) \newline \indent Qiuzhen College, Tsinghua University, Beijing, China}
\email{jiayq24(at)mails(dot)tsinghua(dot)edu(dot)cn}

\author{Xinkai Yu}
\address{(Xinkai Yu) \newline \indent University of Science and Technology of China, 96 Jinzhai Road, Baohe District, Hefei, Anhui Province 230026, P.R. China}
\email{1256981750(at)qq(dot)com}

\author{Zhexi Zhang}
\address{(Zhexi Zhang) \newline \indent Department of Mathematics, Southern University of Science and Technology, 1088 Xueyuan Boulevard, Nanshan District, Shenzhen, Guangdong Province 518055, P.R. China}
\email{zhangzx2022(at)mail(dot)sustech(dot)edu(dot)cn}

\author{Yuchen Zhu}
\address{(Yuchen Zhu) \newline \indent Center for combinatorics, Nankai University, No.\ 94 Weijin Road, Nankai District, Tianjin 300071, P.R.China}
\email{zhych(at)mail(dot)nankai(dot)edu(dot)cn}

\begin{abstract}
We give an explicit, positive, and type-uniform formula for all equivariant structure constants of the Peterson Schubert calculus in arbitrary Lie types, using only the Cartan matrix of the corresponding root system $\Phi$. This solves an open problem originally asked by Harada--Tymoczko in type A for all Lie types. As an application, we derive a type-uniform formula for the mixed $\Phi$-Eulerian numbers. 
\end{abstract}

\maketitle

\setcounter{tocdepth}{2}

\section{Introduction}
The Peterson variety is a remarkable subvariety of the flag variety, introduced by Dale Peterson \cite{peterson1997quantum} to realize the quantum cohomology rings of all (Langlands dual) partial flag varieties geometrically. It can also be used to describe the homology of the affine Grassmannian (see \cite{lam2010quantum}) and it is closely related to the wonderful compactification of a complex semisimple algebraic group \cite{Bualibanu2017Peterson}.
The Peterson varieties in type $A$ are closely related to the Toeplitz matrices \cite{rietsch2003totally}.

Now, let \( G \) be a complex semisimple linear algebraic group with a Borel subgroup \( B \subset G \). Recall that the cohomology ring $H^*(G/B; \mathbb{Z})$ of the flag variety $G/B$ has a distinguished basis consisting of the Schubert classes \cite{bernstein1973schubert}. Studying the product structure of these Schubert classes (sometimes also encompassing the study of analogous questions in generalized cohomology theories or equivariant settings) is the central problem in modern Schubert calculus. By Kleiman's transversality theorem \cite{kleiman1974transversality}, the Schubert structure constants are non-negative integers. In type $A$, the Schubert polynomials in \cite{Lascoux} are polynomial representatives of the Schubert classes in the classical Borel presentation \cite{borel1953cohomologie} of the cohomology ring of the complete flag variety. It is an outstanding problem in combinatorics and enumerative algebraic geometry to give a positive formula or a combinatorial model for the multiplicative structure constants of the Schubert polynomials.

There are also interesting works studying the structure constants of the cohomology ring of the Peterson varieties (the so-called ``Peterson Schubert calculus''). Let $\Pet_G$ denote the Peterson variety inside $G/B$ (see \eqref{eq-def of Pet} for a precise definition). It admits an action of a
one-dimensional sub-torus $S$ of the maximal torus $T\subset G$. Let $\sigma_{v_I} \in H_S^*(G / B)$ be a Schubert class indexed by some Coxeter element $v_I \in W$ for $I \subset \Delta$, where $W$ is the Weyl group and $\Delta$ is the set of simple roots. Let $\iota^*\left(\sigma_{v_I}\right)$ be the equivariant pullback of $\sigma_{v_I}$ along the inclusion $\iota: \Pet_G \rightarrow G / B$. While the class $\iota^*\left(\sigma_{v_I}\right)$ depends on the choice of $v_I$, the class $p_I:=\frac{\iota^*\left(\sigma_{v_I}\right)}{m\left(v_I\right)}$ is independent of this choice after dividing by a certain intersection multiplicity $m\left(v_I\right)$, known as the Peterson Schubert class in the literature. In \cite{Goldin2024positivity}, 
Goldin, Mihalcea, and Singh showed that $\left\{p_I\right\}_{I \subset \Delta}$ is an $H_S^*(p t)$-basis of $H_S^*(\Pet_G; \mathbb{Z})$, and that this basis is dual to the equivariant Borel--Moore homology basis constituted by the fundamental classes of the Peterson Schubert cycles (see Theorem \ref{thm-PSbasis}). 
Furthermore, their main result shows that the structure constants of multiplication $c_{I, J}^K\in H_S^*(p t)\cong\mathbb{Z}[t]$, with respect to the basis $\left\{p_I\right\}_{I \subset \Delta}$, defined by
\begin{equation} \label{eq-def of SC}
p_I \cdot p_J=\sum_K c_{I, J}^K p_K,
\end{equation}
are polynomials in $t$ with non-negative integer coefficients.

A natural open question (asked in \cite[P.43]{Harada2011positive} for Lie type $A$) is the following
\begin{question} \label{ques-intro}
    Is there a positive formula for all structure constants $c_{I, J}^K$ in \eqref{eq-def of SC}?
\end{question}
 In type $A$, Goldin--Gorbutt \cite{goldin2022positive} found positive combinatorial formulae for the structure constants by reducing the general $I,J,K$'s to the consecutive ones successively, while Abe--Horiguchi--Kuwata--Zeng \cite{abe2024geometry} found a different combinatorial model computing these structure constants in non-equivariant cohomology based on a geometric viewpoint. Some years earlier, in type $A$, Harada--Tymoczko \cite{Harada2011positive} gave a positive equivariant Monk formula for the product of a cohomology-degree-2 Peterson Schubert class (that is, $p_{\{i\}}$ for some $i\in \Delta$) with an arbitrary Peterson Schubert class, and Bayegan--Harada \cite{bayegan2013giambelli} gave an equivariant Giambelli formula to express the Peterson Schubert class $p_I$ in terms of the ring generators $p_{\{i\}}$'s. Drellich \cite{drellich2015monk} gave the equivariant Monk rule and Giambelli formula for all Lie types using a type-by-type analysis, and Goldin--Singh \cite{goldin2021equivariant} proved the equivariant Giambelli formula and the equivariant Monk rule using type-independent methods. 

While the positivity of structure constants in the Peterson Schubert calculus was recently established using the aforementioned geometric and combinatorial methods, the above Question \ref{ques-intro} remains unresolved in general, and it is not known whether these constants can be calculated directly and uniformly from the intrinsic Lie-theoretic data of the corresponding root system. This paper addresses this problem. In this paper, we give a simple, explicit, positive, and type-uniform formula for all structure constants in \eqref{eq-def of SC} of the Peterson Schubert calculus in arbitrary Lie types, which completely solves the above question in a type-uniform way. Our formula is a concise algebraic formula, determined only by the Cartan matrix. As far as we know, this is the first example of an explicit, complete, and type-uniform positive algebraic formula in general Schubert calculus. We hope that our method can shed some new light on related structure constants problems in Schubert calculus. We derive this formula from an algebraic viewpoint, which is different from the above-mentioned work. We now present our main results.

Let $C_\Delta:=\left(\left\langle\alpha_i, \alpha_j\right\rangle\right)_{i, j \in \Delta}=(c_{ij})_{i, j \in \Delta}$ be the Cartan matrix associated with the semisimple Lie algebra $\mathfrak{g}$ of $G$. Harada--Horiguchi--Masuda \cite{harada2015equivariant} gave an explicit and beautiful Borel-type presentation of the equivariant cohomology ring of the Peterson varieties in all Lie types; see Theorem \ref{thm-HHMpre}.
We have the following

\begin{theorem}[Peterson Schubert monomials]
Under the isomorphism \eqref{eq-HHMpre}, the Peterson Schubert class $p_I$, $I \subset \Delta$, is represented by the monomial 
\begin{equation} 
\frac{\operatorname{det}\left(C_I\right)}{\left|W_I\right|}\prod_{i \in I} \varpi_i, 
\end{equation}
where $\operatorname{det}\left(C_I\right)$ is the determinant of the Cartan sub-matrix $C_I$ determined by $I\subset\Delta$ and $\left|W_I\right|$ is the order of the parabolic subgroup $W_I$ of the Weyl group $W$ determined by $I\subset\Delta$.    
\end{theorem}

Using the above Peterson Schubert monomials and the Harada--Horiguchi--Masuda presentation \eqref{eq-HHMpre}, we can expand the product of two monomials into the monomial basis representing the Peterson Schubert classes, and compute all the structure constants in \eqref{eq-def of SC}. 
 
To state our main theorem, we first need to introduce some notation. For any $K\subset\Delta$, we use $[C_{K}^{-1}]_{j,k}$ to denote the $(j, k)$-entry of the matrix $C_{K}^{-1}$, which is the inverse of the Cartan sub-matrix $C_{K}$ determined by $K$. 
Suppose $|\Delta|=n$, we use the notation $(\quad)_{J,K\subset\Delta}$ to denote a $2^n\times 2^n$ matrix, with rows indexed by $J\subset\Delta$ and columns indexed by $K\subset\Delta$, where the ordering of the subsets of $\Delta$ is the lexicographical order, and we use the notation $[\quad]_{J,K}$ to denote the $(J, K)$-entry of the matrix $(\quad)_{J,K}$. The following is our main theorem.

\begin{theorem}[Structure constants of equivariant Peterson Schubert calculus] \label{thm-intromain}
With the notation as above. Suppose $I=\{i_1, ..., i_l\} \subset \Delta$. Then the structure constant $c_{I, J}^K$, defined in \eqref{eq-def of SC}, is given by 
$$
c_{I, J}^K=\left[D_{i_1} \cdots D_{i_{\ell}}\right]_{J, K} \frac{\operatorname{det}\left(C_I\right) \operatorname{det}\left(C_J\right)\left|W_K\right|}{\left|W_I\right|\left|W_J\right| \operatorname{det}\left(C_K\right)},
$$
where $D_i:=(d^{K}_{i, J})_{J, K\subset\Delta}$ and
\begin{equation} \label{eq-intro-matrix}
d^{K}_{i, J}=\begin{cases}
\frac{[C_{K}^{-1}]_{i,s}}{[C_{K}^{-1}]_{s,s}}, &\text{ if } i\in K, |K|=|J|+1, \text{ and } K\setminus J=\{s\},\\
2t\sum_{k\in K}[C_K^{-1}]_{i,k}, &\text{ if }i\in K \text{ and }K=J,\\
0, &\text{ otherwise }.
\end{cases}
\end{equation}
\end{theorem}

Since the ordinary cohomology $H^*(\Pet_G; \mathbb{Z})$ vanishes in odd degrees, the Peterson variety $\Pet_G$ is equivariantly formal in the sense of \cite{goresky1998equivariant}. Therefore, by letting $t=0$ in \eqref{eq-intro-matrix}, we get a simple formula for all non-equivariant structure constants with respect to the basis $\left\{p_I\right\}_{I \subset \Delta}$ of $H^*(\Pet_G; \mathbb{Z})$ (we use the same notation to denote the classes under the natural map $H_S^*(\Pet_G; \mathbb{Z}) \rightarrow H^*(\Pet_G; \mathbb{Z})$). 

\begin{theorem} \label{thm-introNmain}
 Suppose $I=\{i_1, ..., i_l\} \subset \Delta$. Then the non-equivariant structure constant $n_{I, J}^K$, defined to be the constant term of $c_{I, J}^K$ in \eqref{eq-def of SC}, is given by 
 $$
n_{I, J}^K=\left[M_{i_1} \cdots M_{i_{\ell}}\right]_{J, K} \frac{\operatorname{det}\left(C_I\right) \operatorname{det}\left(C_J\right)\left|W_K\right|}{\left|W_I\right|\left|W_J\right| \operatorname{det}\left(C_K\right)},
$$ where $M_i:=(m^{K}_{i, J})_{J, K\subset\Delta}$ and 
\begin{equation} \label{eq-intro-Nmatrix}
m^{K}_{i, J}=\begin{cases}
\frac{[C_{K}^{-1}]_{i,s}}{[C_{K}^{-1}]_{s,s}}, &\text{ if } i\in K, |K|=|J|+1, \text{ and } K\setminus J=\{s\},\\
0, &\text{ otherwise }.
\end{cases}
\end{equation}
\end{theorem}

To prove the above theorems, we transform the structure constants in \eqref{eq-intro-matrix} into an infinite matrix series (see the proof of Theorem \ref{thm-main}) related to the \emph{Coxeter adjacency matrix} $A_K$, defined as $A_K:=2E-C_K$, where $E$ is the identity matrix. We use the classical \emph{Perron--Frobenius theorem} to prove that this matrix series is convergent and related to the inverse of a matrix; see Lemma \ref{lem: convergence of matrix series}. Finally, we use the \emph{Sherman--Morrison formula} to compute the matrix inverse, which is simplified to the final formula; see Lemma \ref{lem: manifestly positive formula}.

It is well known that the entries of the inverse of any Cartan matrix are all non-negative (see, for example, \cite[Exercise 13.8]{humphreys1972}), hence the entries of all the matrices $(d^{K}_{i, J})_{J, K}$ in Theorem \ref{thm-intromain} are polynomials in $t$ with non-negative coefficients.
Since the entries in the product of such matrices remain polynomials in $t$ with non-negative coefficients, our formula is indeed a positive formula, which provides a new purely algebraic proof of the positivity of equivariant structure constants of the  Peterson Schubert calculus. We also provide a simple criterion for when these structure constants are non-zero; see Corollary \ref{cor-nonzero} and Remark \ref{rmk-nonzero}.

The relationship between the (non-equivariant) Peterson Schubert calculus and the mixed Eulerian numbers (which
were introduced and studied by Postnikov
\cite{postnikov2009permutohedra}) has been actively explored very recently; see \cite{Berget2023log}, \cite{Nadeau2023permutahedral}, and \cite{Horiguchi2024mixed}. These numbers compute the mixed volumes of $\Phi$-hypersimplices, see Section \ref{sec-Applications} for the precise definitions. The origins of this connection lie in the realization of the Peterson variety as a flat degeneration of a smooth projective toric variety, the $W$-permutohedral variety; see, for example, 
\cite{Abe2020Geometry,Klyachko1985orbits,Klyachko1995toric}. In \cite{Horiguchi2024mixed}, Horiguchi provided a combinatorial model introduced in \cite{abe2024geometry} for the computation of the mixed $\Phi$-Eulerian numbers when $\Phi$ is of type $A$ and derived a type-by-type computation for the mixed $\Phi$-Eulerian numbers for general Lie types by iteratively applying the (non-equivariant version of) Monk formula from Drellich in \cite{drellich2015monk}. As an application of our main theorem, we derive a type-uniform formula for the mixed $\Phi$-Eulerian numbers in arbitrary Lie type.

\begin{theorem}\label{thm-introMEN}
    Let \(\Phi\) be an irreducible root system of rank $n$. Let \(c_1, \ldots, c_n\) be non-negative integers with \(c_1 + \cdots + c_n = n\).  
The mixed ${\Phi}$-Eulerian number $A_{c_{1}, \dots, c_{n}}^{\Phi}$ can be computed using the following formula:
\[
A_{c_{1},\ldots,c_{n}}^{\Phi} = \frac{|W_{\Phi}|}{\det(C_\Phi)} 
\left[ M_1^{c_{1}} \cdots M_n^{c_{n}} \right]_{\emptyset, \Delta}.
\]
where $M_i$ is the matrix defined in \eqref{eq-intro-Nmatrix}.
The notation
\(
[\quad]_{\emptyset, \Delta}
\)
denotes the entry in the row indexed by $\emptyset$ and the column indexed by $\Delta$.
\end{theorem}

Compared with the computations in \cite{Horiguchi2024mixed}, our formula for the mixed $\Phi$-Eulerian numbers is quite simple and direct, which avoids the need to discuss the changes in the Lie types case-by-case when considering sub-root systems as in \cite[Section 7]{Horiguchi2024mixed}.

The remainder of the paper is organized as follows. In Section \ref{sec-preli}, we introduce the notation and some preliminary results which we use throughout this paper. In Section \ref{sec-Structure}, we prove our main theorems. In Section \ref{sec-Applications}, we give the application to the mixed Eulerian numbers. 

\subsection*{Acknowledgments}
The first author would like to express his gratitude to Hiraku Abe and Haozhi Zeng for helpful discussions. 
An extended abstract of this paper will appear in the Proceedings of the 38th Conference on Formal Power Series and Algebraic Combinatorics (Seattle), the authors thank Yibo Gao for some useful comments and thank anonymous referees of the extended abstract for suggestions which helped improve the exposition. Computer calculations in service of this paper were coded in SageMath. 
The first author is partially supported by NSFC Grant No. 12471309.

\section{Preliminaries} \label{sec-preli}

In this section, we collect our notation and preliminary results for later use. 

\subsection{Flag varieties and Schubert varieties} \label{sec-pre-flag}
Let $G$ be a complex connected, simply connected, semisimple algebraic group of rank $n$, $B\subset G$ be a Borel subgroup, and $B^{-}\subset G$ be the opposite Borel.  Write $T:=B \cap B^{-}$ for the maximal torus common to both Borel subgroups. Let $\mathfrak{g}=\Lie(G)$, $\mathfrak{b}=\Lie(B)$, and $\mathfrak{h}=\Lie(T)$ be the corresponding Lie algebras. Denote by $\Phi$ the root system determined by $(G,T)$, and by
$\Delta=\{\alpha_{1},\dots,\alpha_{n}\}$ the simple roots specified by $(B,T)$. We abuse the notation by letting $\Delta$ also denote the index set of simple roots. The Weyl group $W:=\langle s_{i}\mid i\in\Delta\rangle$ is a finite Coxeter group. For any subset $I\subset\Delta$, we have the standard parabolic subgroup $P_{I}\subset G$,
whose Weyl group is $W_{I}= \langle s_{i}\mid i\in I\rangle$, a parabolic subgroup of $W$.

Recall that the group \(G\) admits the Bruhat decomposition
\[
G=\bigsqcup_{w\in W} BwB,
\]
a disjoint union of \(B\)-double cosets indexed by the Weyl group elements. 
Passing to the flag variety, we obtain
\begin{equation} \label{eq-Bruhat}
    G/B=\bigsqcup_{w\in W} BwB/B,
\end{equation}
where each \(BwB/B\) is isomorphic to an affine space, which is called the \emph{Schubert cell} \(X^{\circ}_{w}:=BwB/B\).

Similarly, we have the opposite Bruhat decomposition
\begin{equation} \label{eq-OBruhat}
G/B=\bigsqcup_{w\in W} B^-wB/B,
\end{equation}
where each \(B^-wB/B\) is again isomorphic to an affine space, called the \emph{opposite Schubert cell} \(\Omega^{\circ}_{w}:=B^-wB/B\).

For every \(w\in W\), we set
\[
X_{w}:=\overline{X^{\circ}_{w}}=\overline{BwB/B}
\quad\text{and}\quad
\Omega_{w}:=\overline{\Omega^{\circ}_{w}}=\overline{B^{-}wB/B},
\]
the Schubert and opposite Schubert varieties in \(G/B\), respectively. It follows from \eqref{eq-OBruhat} that the equivariant cohomology classes $\sigma_w \in H_T^*(G/B;\mathbb{Z})$, which are Poincaré dual to the fundamental class $\left[\Omega_{w}\right]$, form a basis $\left\{\sigma_w \mid w \in W\right\}$ of $H_T^*(G / B;\mathbb{Z})$ as a module over $H_T^*(pt)$, called the \emph{equivariant Schubert classes}.

\subsection{Peterson varieties and some geometric constructions} \label{sec-pre-Peterson}
We recall the definition of the Peterson variety in any Lie type and provide some related geometric constructions from the existing literature.

Recall that we have the Cartan decomposition of $\mathfrak{g}$ into root spaces
$$\mathfrak{g}=\mathfrak{h}\oplus\bigoplus_{\alpha\in\Phi}\mathfrak{g}_{\alpha}.$$
Fix a regular nilpotent element $e\in\mathfrak{b}$.  The \emph{Peterson variety} $\Pet_{G}$ is a closed subvariety of the flag variety $G/B$, defined by
\begin{equation}\label{eq-def of Pet}
\Pet_{G}:=\Bigl\{gB\in G/B\;\Big|\;
\text{Ad}_{g^{-1}}(e)\in\mathfrak{b}\oplus
{\textstyle\bigoplus\limits_{\alpha\in\Delta}}\mathfrak{g}_{-\alpha}\Bigr\}.
\end{equation}
It is known that the Peterson variety $\Pet_{G}$ is always irreducible with complex dimension $n:=\rank G$ (see \cite{peterson1997quantum} or \cite[Proposition 6.2]{lam2016total}), but it is singular and even non-normal in general \cite{kostant1996flag}. The Peterson varieties also form a special class of a larger family of subvarieties of the flag variety, namely the \emph{Hessenberg varieties} introduced in \cite{de1992hessenberg}; see \cite{abe2017survey} for a very nice survey.

\begin{eg}[Type $A_{n-1}$ Peterson variety] \label{ex:typeA}
In Lie type $A$, the Peterson variety can be easily described in a concrete way. Let $n\ge 1$ be an integer and let
\[
\mathrm{Fl}_n:=\bigl\{V_\bullet=(0=V_0\subsetneq V_1\subsetneq \dots \subsetneq V_n=\mathbb C^n)\mid \dim_{\mathbb C}V_i=i\bigr\}
\]
denote the full flag variety of $\mathbb C^n$.  Fix the regular nilpotent matrix
\[
N:=
\begin{pmatrix}
0&1&0&\cdots&0\\
0&0&1&\cdots&0\\
\vdots&\vdots&\ddots&\ddots&\vdots\\
0&0&\cdots&0&1\\
0&0&\cdots&0&0
\end{pmatrix}\in\mathbb C^{n\times n}.
\]
The Peterson variety $\mathrm{Pet}_n\subset\mathrm{Fl}_n$ is defined by
\[
\mathrm{Pet}_n:=\bigl\{V_\bullet\in\mathrm{Fl}_n\mid NV_i\subset V_{i+1}\ \text{for all}\ 1\le i\le n-1\bigr\},
\]
where $NV_i$ denotes the image of $V_i$ under the linear map $N\colon\mathbb C^n\to\mathbb C^n$.
In this setting, the Peterson varieties are also known as the \emph{regular nilpotent} Hessenberg varieties with \emph{Hessenberg function} $h(i)=i+1$ for $i=1,2, \ldots, n-1$.
\end{eg}

It is known that there is a one-dimensional sub-torus $S\subset T$ acting on the Peterson variety $\Pet_G$ and we set $H_S^*(p t)\cong\mathbb{Z}[t]$, where $t$ is a certain character of $S$; see \cite[Section 3.2]{Goldin2024positivity}. By intersecting the Peterson variety $\Pet_G$ with Schubert cells and opposite Schubert cells in the Bruhat decomposition \eqref{eq-Bruhat} and opposite Bruhat decomposition \eqref{eq-OBruhat} of the flag variety $G/B$, we have the following 

\begin{proposition}[\protect{See \cite[Proposition 3.3]{Goldin2024positivity} and \cite[Proposition 6.2]{lam2016total}}]
    The following are equivalent:
\begin{enumerate}
    \item $\Pet_G \cap X_w^\circ \neq \emptyset$;
    \item $\Pet_G \cap  \Omega_w^\circ\neq \emptyset$;
    \item $w = w_I$ for some $I \subset \Delta$,
where $w_{I}$ is the longest element of $W_{I}$.
\end{enumerate}
\end{proposition}

The geometry of the intersection $\Pet_G  \cap \Omega_{w_{I}}^\circ$ encodes the (small) quantum cohomology ring $qH^*(G^{\vee} / P_I^{\vee})$ of the Langlands dual partial flag variety $G^{\vee} / P_I^{\vee}$, which was Dale Peterson's original motivation in his unpublished work \cite{peterson1997quantum} to introduce and study these remarkable varieties. Furthermore, it is known that the set-theoretic intersection $\Pet_{G,I}^{\circ}:=\Pet_G \cap X_{w_I}^\circ$ is an affine space of dimension $|I|$ (see \cite{tymoczko2007paving} or \cite{Bualibanu2017Peterson}), called a \emph{Peterson cell}. The Bruhat decomposition \eqref{eq-Bruhat} induces an $S$-stable affine paving 
\begin{equation} \label{eq-paveP}
\Pet_{G}=\bigsqcup_{I \subset \Delta} \Pet_{G, I}^{\circ}.
\end{equation}

For $I\subset\Delta$, we set
\begin{equation} \label{eq-PSV}
  \Pet_{G, I}:=\overline{\Pet_{G, I}^{\circ}}=\Pet_{G} \cap X_{w_{I}},\qquad
\Omega_{I}:=\Pet_{G} \cap\Omega_{w_{I}},  
\end{equation}
where $\Pet_{G, I}$ is known as the \emph{Peterson Schubert variety} in the literature.
By the above affine paving \eqref{eq-paveP}, we have the following theorem; see \cite{peterson1997quantum} and \cite{tymoczko2007paving}.

\begin{proposition}
For every $I\subset\Delta$, we have the fundamental class $[\Pet_{G,I}]$ in $H^S_{2|I|}(\Pet_{G};\mathbb{Z})$, and the collection
$\{[\Pet_{G,I}]\mid I\subset\Delta\}$ forms a basis of the total equivariant homology $H^S_{*}(\Pet_{G};\mathbb{Z})$.
\end{proposition}

\subsection{Cohomology ring of Peterson varieties and Peterson Schubert classes} \label{sec-pre-Cohomology}

Write $B=T\ltimes U$ with $U$ the unipotent radical. The canonical projection $\pi\colon B\twoheadrightarrow T$ sends $b=tu\mapsto t$.  Any weight $\mu\colon T\to\mathbb{C}^{\times}$ extends to a character of $B$ by pulling back along $\pi$. Denote the associated one-dimensional $B$-module by $\mathbb{C}_{\mu}$ and consider the line bundle
\[
L_{\mu}:=G\times^{B}\mathbb{C}_{-\mu}.
\]
Sending $\mu$ to the equivariant first Chern class $c_{1}(L_{\mu})$ gives the homomorphism
\[
\Lambda=\bigoplus_{i=1}^{n}\mathbb{Z}\varpi_{i}\rightarrow
H_T^{2}(G/B;\mathbb{Z}),
\qquad \mu\longmapsto c_{1}(L_{\mu}),
\]
where $\Lambda=\bigoplus_{i=1}^{n}\mathbb{Z}\varpi_{i}$ is the weight lattice. We identify all weights $\mu$ with $c_{1}(L_{\mu})$ in cohomology by abusing notation. 

Consider the restriction of the line bundle $L_{\mu}$ on the Peterson variety $\Pet_{G}$, which will also be denoted by $L_{\mu}$ when there is no confusion.  Harada--Horiguchi--Masuda \cite{harada2015equivariant} gave the following explicit and beautiful Borel-type presentation of the equivariant cohomology ring of the Peterson varieties in all Lie types. 

\begin{theorem}[\cite{harada2015equivariant}, Theorem 4.1] \label{thm-HHMpre}
    Suppose that rank of $G$ is $n$. The $S$-equivariant cohomology ring of the Peterson variety $\Pet_G$ has the following presentation
    \begin{equation}\label{eq-HHMpre}
H_S^*(\Pet_G; \mathbb{Q}) \cong \frac{\mathbb{Q}[\varpi_1, \ldots, \varpi_n,t]}{\left(\sum_{j=1}^n \left\langle \alpha_i, \alpha_j \right\rangle \varpi_i \varpi_j-2t\varpi_i \mid 1 \leq i \leq n \right)}.
\end{equation}
\end{theorem}

For the above $\mathbb{Q}$-coefficient, see their final remark in \cite{harada2015equivariant}. In the above Harada--Horiguchi--Masuda presentation \eqref{eq-HHMpre}, the variable $\varpi_{i}$ corresponds to 
the equivariant first Chern class $c_{1}(L_{\varpi_{i}})\in H_S^{2}(\Pet_{G};\mathbb{Q})$ of the line bundle $L_{\varpi_{i}}$ on $\Pet_{G}$ associated with the fundamental weight $\varpi_{i}$; see, for example, \cite[Section 6.5]{abe2023peterson}.

For $I \subset \Delta$, recall that an element $v \in W$ is called a \emph{Coxeter element} for $I$ if $v=s_{\alpha_1} \cdots s_{\alpha_k}$ for some enumeration $\alpha_1, \ldots, \alpha_k$ of $I$. Let $\iota: \Pet_{G} \hookrightarrow G / B$ denote the closed embedding from the Peterson variety to the flag variety. In \cite{Goldin2024positivity}, a basis $\left\{p_I\right\}_{I \subset \Delta}$ of $H_S^*(\Pet_{G};\mathbb{Z})$ dual to the basis $\left\{\left[\Pet_{G,I}\right]\right\}_{I \subset \Delta}$ of $H^S_*(\Pet_{G};\mathbb{Z})$ is constructed. We call $p_I$ an (integral equivariant) \emph{Peterson Schubert class}.

\begin{theorem}[\protect{\cite[Theorem 4.3 and Corollary 4.4]{Goldin2024positivity}}] \label{thm-PSbasis}
    For each $I \subset \Delta$, fix a Coxeter element $v_I$, and consider the pull-back of the equivariant Schubert class $\iota^* \sigma_{v_I} \in H_S^*(\Pet_{G};\mathbb{Q})$. Then the classes 
    \begin{equation}\label{eq-PSclass}
      \left\{\left.p_I:=\frac{\iota^* \sigma_{v_I}}{m(v_I)} \in H_S^*(\Pet_{G};\mathbb{Q}) \right\rvert\, I \subset \Delta\right\} 
    \end{equation}
    form a $H_S^*(p t)$-basis of $H_S^*(\Pet_{G};\mathbb{Z})$, which is dual to the basis $\left\{\left[\Pet_{G,I}\right]\right\}_{I \subset \Delta}$ of $H^S_*(\Pet_{G};\mathbb{Z})$ under the standard pairing between equivariant homology and equivariant cohomology, where $m\left(v_I\right)$ is the multiplicity (which is a positive integer) of the unique point $w_I$ in the intersection $\Omega_{w_I} \cap \Pet_{G,I}$.
\end{theorem}

It follows that while the class $\iota^*\left(\sigma_{v_I}\right)$ depends on the choice of $v_I$, the class $p_I$ is independent of this choice. In \cite[Remark 7.7]{Goldin2024positivity}, a type-independent formula for the intersection multiplicity $m\left(v_I\right)$ was conjectured, which was proved in \cite{goldin2021equivariant}.

\begin{theorem}[\protect{Intersection multiplicity formula, \cite[Theorem 1.3]{goldin2021equivariant}}]
Suppose I is a connected Dynkin diagram. Let $v_I$ be a Coxeter element of $I$, and let $R\left(v_I\right)$ denote the number of reduced expressions for $v_I$. We have
\begin{equation} \label{eq-IMF}
m\left(v_I\right)=\frac{R\left(v_I\right)\left|W_I\right|}{|I|!\operatorname{det}\left(C_I\right)},
\end{equation}
where $W_I$ is the Weyl group determined by $I$, and $C_I$ is the Cartan matrix determined by $I$.
\end{theorem}

Next, we recall the equivariant Giambelli formula for the Peterson variety in \cite{drellich2015monk} and \cite{goldin2021equivariant}. It is well known that the equivariant Schubert classes $\sigma_{s_i}, i\in\Delta$, generate $H_T^*(G / B;\mathbb{Q})$ as a $H_T^*(p t)$-algebra. Since the restriction map $H_T^*\left(G/B ; \mathbb{Q}\right) \rightarrow H_S^*\left(\Pet_{G} ; \mathbb{Q}\right)$ is surjective \cite{harada2015equivariant}, it is easy to see that each $\iota^* \sigma_{v_I}$ can be expressed as a polynomial with $\mathbb{Q}[t]$-coefficients in $\iota^* \sigma_{s_i}, i\in\Delta$. Drellich  \cite{drellich2015monk} gave a very simple equivariant Giambelli formula for this polynomial for a particular choice of $v_I$, while Goldin--Singh \cite{goldin2021equivariant} gave a type-independent proof to show that this formula works for all Coxeter elements.

\begin{theorem}[\protect{Equivariant Giambelli formula for Peterson varieties, \cite[Theorem 1.2]{goldin2021equivariant}}]
Suppose I is a connected Dynkin diagram. Let $v_I$ be a Coxeter element for $I$, let $R\left(v_I\right)$ be the number of reduced words for $v_I$, then we have
\begin{equation}\label{eq-Giambelli}
   \iota^* \sigma_{v_I}=\frac{R\left(v_I\right)}{|I|!}\prod_{i \in I} \iota^* \sigma_{s_i}.
\end{equation}
\end{theorem}

\section{Structure constants of Peterson Schubert calculus} \label{sec-Structure}
For $I\subset\Delta$, we define
\begin{equation} \label{eq-monomial}
    \varpi_{I}:=\prod_{i\in I}\varpi_{i}\in H_S^{2|I|}(\Pet_{G};\mathbb{Q}).
\end{equation}

Firstly, we identify the Peterson Schubert classes with our \emph{Peterson Schubert monomials} under the Harada--Horiguchi--Masuda presentation \eqref{eq-HHMpre}. 

\begin{theorem}[Peterson Schubert monomials] \label{thm-PSM}
Under the isomorphism \eqref{eq-HHMpre}, the Peterson Schubert class $p_I$, $I \subset \Delta$, defined in \eqref{eq-PSclass}, is represented by the monomial 
\begin{equation} \label{eq-PSM}
\frac{\operatorname{det}\left(C_I\right)}{\left|W_I\right|}\varpi_{I}, 
\end{equation}
where $\operatorname{det}\left(C_I\right)$ is the determinant of the Cartan sub-matrix $C_I$ determined by $I\subset\Delta$, and $\left|W_I\right|$ is the order of the parabolic subgroup $W_I$ of the Weyl group $W$ determined by $I\subset\Delta$.
\end{theorem}

\begin{proof}
   As is well-known, the equivariant Schubert class $\sigma_{s_i}$ is equal to $\varpi_{i}$, the equivariant first Chern class of the line bundle $L_{\varpi_{i}}$ on the flag variety $G/B$; see \cite{bernstein1973schubert} and \cite{harada2015equivariant}. By the naturality of the Chern class, we have 
   \begin{equation} \label{eq-S=C}
      \iota^* \sigma_{s_i}=\varpi_i \quad \text { in } H_S^{2}(\Pet_{G};\mathbb{Q}). 
   \end{equation}
   Suppose first that $I$ is a connected Dynkin diagram, then by \eqref{eq-PSclass}, \eqref{eq-IMF}, \eqref{eq-Giambelli}, \eqref{eq-monomial}, and \eqref{eq-S=C}, we have
 
       \begin{align*}
p_I&=\frac{\iota^* \sigma_{v_I}}{m\left(v_I\right)}\\
&=\frac{|I|!\operatorname{det}\left(C_I\right)}{R\left(v_I\right)\left|W_I\right|}\iota^* \sigma_{v_I}\\
&=\frac{|I|!\operatorname{det}\left(C_I\right)}{R\left(v_I\right)\left|W_I\right|}\frac{R\left(v_I\right)}{|I|!} \prod_{i \in I} \iota^* \sigma_{s_i}\\
&=\frac{\operatorname{det}\left(C_I\right)}{\left|W_I\right|}\prod_{i \in I}\varpi_i\\
&=\frac{\operatorname{det}\left(C_I\right)}{\left|W_I\right|} \varpi_I.
       \end{align*}

For a general Dynkin diagram $I$, notice that $\sigma_{v_I}$, $\operatorname{det}\left(C_I\right)$, $W_I$, and $R\left(v_I\right)$ all factor multiplicatively with respect to the decomposition of connected components of $I$, hence the conclusion follows from the connected case. 
\end{proof}

\begin{remark}
    By letting $t=0$, the above arguments also give a direct proof of \cite[Theorem 1.1.(3)]{Horiguchi2024mixed} without going through the mixed $\Phi$-Eulerian numbers.
\end{remark}

\begin{remark}
    Actually, by the Whitney sum formula, the monomial $\varpi_I$ expresses the equivariant Euler class $e\left(V_I\right)$ of the vector bundle 
    $$
V_I:=\bigoplus_{i \in I} L_{\varpi_i}
$$
on the Peterson variety $\Pet_{G}$. This vector bundle admits a well-defined section; see \cite[Section 3.2]{abe2024geometry} and \cite[Section 5.2]{abe2023peterson}. The zero locus of this section is exactly $\Omega_{I}:=\Pet_{G}\cap\Omega_{w_{I}}$ defined in \eqref{eq-PSV}; see \cite[Proposition 3.12]{abe2024geometry} and \cite[Section 5.4]{abe2023peterson}. Since $\Pet_{G,I} \cap \Omega_I=\left\{w_I\right\}$, it is not hard to see that $\left\{\varpi_I\right\}_{I \subset \Delta}$ is dual to the basis $\left\{\left[\Pet_{G,I}\right]\right\}_{I \subset \Delta}$ of $H^S_*(\Pet_{G};\mathbb{Z})$ up to a constant. This constant can be computed via integration over the Peterson Schubert variety $\Pet_{G,I}$, which is actually a product of smaller Peterson varieties; see \cite[Section 4.2] {abe2024geometry} for the computations in type $A$. This would give a different proof of the above theorem.
\end{remark}

Let $d_{I,J}^K$ denote the structure constant of $\{\varpi_{I}\}_{I\subset\Delta}$,
determined by 
$$\varpi_I \varpi_J=\sum_{\substack{K\subset \Delta\\}} d^{K}_{I,J}\varpi_K,$$
By Theorem \ref{thm-PSM}, we have the following immediate corollary.

\begin{corollary}\label{cor-relation}
    The relation between the structure constants $c_{I,J}^K$ and $d_{I,J}^K$ is
\[
c_{I,J}^K = \frac{\det(C_I) \det(C_J) |W_K|}{|W_I| |W_J| \det(C_K)} d_{I,J}^K.
\]

\end{corollary}

For any $K\subset\Delta$, let $A_K:=2E-C_K$, known as the \emph{Coxeter adjacency matrix} in the literature, where $E$ is the identity matrix. Let $B_K:=\frac{1}{2}A_K$, and $B_K^{\widehat{s}}$ be $B_K$ with the entries in the row indexed by $s$ being zeros. We have the following theorem.

\begin{theorem}\label{thm-main} Suppose $I=\{i_1, ..., i_l\} \subset\Delta$. Then the matrix $(d^K_{I, J})_{J, K}$ of structure constants is equal to the product of matrices $(d^{K}_{i_1, J})_{J, K}\cdots (d^{K}_{i_l, J})_{J, K}$, where 
\begin{equation}\label{eq-def of matrix}
d^{K}_{i, J}=\begin{cases}
    1, &\text{ if } K=J \bigsqcup \{i\},\\
    b^{K, \widehat{s}}_{i, s}, &\text{ if } i\in J, |K|=|J|+1, \text{ and } K\setminus J=\{s\},\\
    2t\sum_{k\in K}[C_K^{-1}]_{i,k}, &\text{ if } i\in K \text{ and } K=J, \\
     0, &\text{ otherwise }.
\end{cases}
\end{equation}
Here $b^{K, \widehat{s}}_{i, s}$ denotes the entry in the row indexed by $i$ and the column indexed by $s$ of
$B_K^{\widehat{s}}\left(E-B_K^{\widehat{s}}\right)^{-1}$ and $[C_K^{-1}]_{i,s}$ denotes the $(i,s)$-entry of $C_K^{-1}$.
\end{theorem}

\begin{proof}
Recall that $\varpi_I=\prod_{m=1}^l\varpi_{i_m}$. To calculate the product $\varpi_I\varpi_J$, we can multiply each element $\varpi_{i_m}$ by $\varpi_J$, which can be represented by a matrix according to the ordered basis of all subsets of $\Delta$. Then the matrix $(d^K_{I, J})_{J, K}$ of structure constants $d_{I,J}^K$ will be equal to the product of these matrices. Without loss of generality, we can consider the calculation of $d_{i,J}^K$.

For $i\notin J$,
\[ \varpi_{i}\varpi_{J}=\varpi_{J\bigsqcup\{i\} },
    \]
so we have $d_{i,J}^{J\bigsqcup\{i\}}=1$.

    It remains to compute the structure constants for $i\in J$.
    Recall that we have the relation 
    $$\sum_{j\in \Delta}c_{ij}\varpi_i\varpi_j-2t\varpi_i=0$$ 
    in the presentation \eqref{eq-HHMpre} of the equivariant cohomology ring $H_S^*(\Pet_G)$ of the Peterson variety. Since $c_{i,i}=2$, for any $i\in \Delta$, it is equivalent to
    \[ \varpi_i^2=\sum_{j \neq i}b_{i,j} \varpi_i\varpi_j+t\varpi_i,
    \]
    where $b_{i, j}$ is the $(i,j)$-entry of  $B_{\Delta}=E-\frac{1}{2}C_\Delta$. Now, let us compute
\begin{equation}\label{eq-computation}
\begin{aligned}
&\varpi_i\varpi_{J}\\=
&\left(\sum_{a_1\neq i}b_{i,a_1}\varpi_i\varpi_{a_1}+t\varpi_i\right)\varpi_{J\setminus\{i\}} \\
        =&\sum_{a_1\in J\setminus\{i\}}b_{i,a_1}\varpi_{a_1}\varpi_{J}+\sum_{s\notin J}b_{i,s}\varpi_{J\cup\{s\}} +t\varpi_J\\
        =&\sum_{a_1\in J\setminus\{i\}} b_{i,a_1}\left(\sum_{a_2\neq a_1}b_{a_1,a_2}\varpi_{a_1}\varpi_{a_2}+t\varpi_{a_1}\right)\varpi_{J\setminus\{a_1\}}+\sum_{s\notin J}b_{i,s}\varpi_{J\cup\{s\}}+t\varpi_J\\
        =&\sum_{a_1\in J\setminus\{i\}}\sum_{a_2\in J\setminus\{a_1\}}b_{i,a_1}b_{a_1,a_2}\varpi_{a_2}\varpi_J\\&+\left(\sum_{s\notin J}b_{i, s}+\sum_{a_1\in J\setminus\{i\}}\sum_{s\notin J}b_{i,a_1}b_{a_1,s} \right)\varpi_{J\cup\{s\}}+\left(1+\sum_{a_1\in J\setminus\{i\}}b_{i,a_1}\right)t\varpi_J \\
        =&\cdots\cdots \\
        =&\sum_{a_m\in J\setminus\{a_{m-1}\}}b_{i,a_m}^{J,m}\varpi_{a_m}\varpi_J+\sum_{s\notin J}\left(b_{i,s}^{J,1}+b_{i,s}^{J,2}+\cdots+b_{i,s}^{J,m}\right)\varpi_{J\cup\{s\}}\\
        &+(F_i^{J,0}+\cdots+F_i^{J,m-1})t\varpi_J,
\end{aligned}
\end{equation}
where 

    \begin{align*}
&b_{i,s}^{J,1}=
b_{i,s},\\
&b_{i,s}^{J,2}=\sum_{a_1\in J\setminus\{i\}}
b_{i,a_1}b_{a_1,s},\\
&\cdots\\
&b_{i,s}^{J,m}=\sum_{a_1\in J\setminus\{i\}}\sum_{a_2\in J\setminus\{a_1\}}\cdots\sum_{a_{m-1}\in J\setminus\{a_{m-2}\}}
b_{i,a_1}b_{a_1,a_2}\cdots b_{a_{m-1},s},\\
&b_{i,a_m}^{J,m}=\sum_{a_1\in J\setminus\{i\}}\sum_{a_2\in J\setminus\{a_1\}}\cdots\sum_{a_{m-1}\in J\setminus\{a_{m-2}\}}
b_{i,a_1}b_{a_1,a_2}\cdots b_{a_{m-1},a_m},\\
&F_i^{J,0}=1,\\
&F_i^{J,1}=\sum_{a_1\in J\setminus\{i\}}b_{i,a_1},\\
&\cdots\\
&F_i^{J,m-1}
=\sum_{a_1\in J\setminus\{i\}}\cdots\sum_{a_{m-1}\in J\setminus\{a_{m-2}\}}
b_{i,a_1}b_{a_1,a_2}\cdots b_{a_{m-2},a_{m-1}}.
    \end{align*}

To finish the proof of the above theorem, it suffices to prove the following lemmas.

\begin{lemma} \label{lemma-matrixpower}
Let $J,K\subset \Delta$, $i\in J$, $K=J \bigsqcup \{s\}$. Then for any $m\geq 1$ and $k\in K$, we have
\begin{equation*}
[(B_{K}^{\widehat{s}})^m]_{i,k}=\sum_{a_1\in J\setminus\{i\}}\sum_{a_2\in J\setminus\{a_1\}}\cdots\sum_{a_{m-1}\in J\setminus\{a_{m-2}\}}
b_{i,a_1}b_{a_1,a_2}\cdots b_{a_{m-1},k}.
\end{equation*}
\end{lemma}

\begin{proof}
    Recall that $B_K^{\widehat{s}}$ is defined as $B_K:=\frac{1}{2}A_K$ with the entries in the row indexed by $s$ being zeros. Since $A_K:=2E-C_K$ and $c_{ii}=2$ for any $i\in\Delta$, we have $b_{i,i}=0$ for any $i\in K$. Hence, whenever $a_{i}=a_{i-1}$ for some $i\in[m-1]$,  
    \[
    b_{i,a_1} b_{a_1,a_2} \cdots b_{a_{m-1},k} = 0,
    \]
    where we take the convention $a_0:=i$.
Therefore, we have
    \begin{align*}
        [(B_{K}^{\widehat{s}})^m]_{i,k}&=\sum_{a_1\in K\setminus\{s\}}\sum_{a_2\in K\setminus\{s\}}\cdots\sum_{a_{m-1}\in K\setminus\{s\}} b_{i,a_1}b_{a_1,a_2}\cdots b_{a_{m-1},k}\\
        &=\sum_{a_1\in K\setminus\{i,s\}}\sum_{a_2\in K\setminus\{a_1,s\}}\cdots\sum_{a_{m-1}\in K\setminus\{a_{m-2},s\}} b_{i,a_1}b_{a_1,a_2}\cdots b_{a_{m-1},k}\\
        &=\sum_{a_1\in J\setminus\{i\}}\sum_{a_2\in J\setminus\{a_1\}}\cdots\sum_{a_{m-1}\in J\setminus\{a_{m-2}\}}
b_{i,a_1}b_{a_1,a_2}\cdots b_{a_{m-1},k}.
    \end{align*}
\end{proof}

A subset $K \subset \Delta$ is called \emph{connected} if the induced Dynkin diagram with the set of vertices $K$ is a connected graph. We also need the following lemma.

\begin{lemma}\label{lem: convergence of matrix series} Let $K\subset\Delta$. Then for any $s\in K$,  all eigenvalues of the matrix $B_K^{\widehat{s}}$ have absolute values less than 1. As a consequence, the matrix $E-B_K^{\widehat{s}}$ is invertible and we have
    \begin{equation}\label{eq-convergence}
        B_K^{\widehat{s}}\left(E-B_K^{\widehat{s}}\right)^{-1}=\sum_{m=1}^{\infty}\left(B_K^{\widehat{s}}\right)^m.
    \end{equation}
\end{lemma}
\begin{proof}
    It suffices to prove the lemma for connected $K$, since otherwise the matrix $B_K^{\widehat{s}}$ is block diagonal. Now suppose $K$ is a connected subset of $\Delta$, it is shown in \cite[Lemma 2]{damianou2014characteristicpolynomialcartanmatrices} that the eigenvalues of the Coxeter adjacency matrix $A_{K}$ are
    \[2\cos{\frac{m_i\pi }{h}},\]
    where $m_i$'s are the exponents of the simple Lie algebra $\mathfrak{g}$ associated with $K$ and $h$ is the Coxeter number of $\mathfrak{g}$. It follows immediately that all eigenvalues of the matrix $A_K$
    have absolute values less than 2, since all $m_i$'s are natural numbers less than $h$. Recall that $B_K^{\widehat{s}}$ is defined as $B_K:=\frac{1}{2}A_K$ with the entries in the row indexed by $s$ being zeros. 
    We can decompose this procedure into two steps: the first step is to multiply by $\frac{1}{2}$, and the second step is to set the entries in the row indexed by $s$ to zeros. The first step will halve all the eigenvalues, which is exactly what we want---absolute values less than 1. We claim that the second step can only keep or decrease the absolute values of the eigenvalues. Since for $i\neq j$, the off-diagonal entries $c_{ij}$ of the Cartan matrix are all non-positive, it is easy to see that $B_K^{\widehat{s}}$ is a non-negative matrix for any $K\subset\Delta$ and $s\in K$. The claim follows from the classical Perron--Frobenius theorem in \cite[Theorem 8.8.1(b)]{algebraicgraphtheoryperronfreobenius}, because we have assumed $K$ to be connected, and then the underlying directed graph of $K$ will be strongly connected. The proof of the first part of the lemma is completed, and the consequence follows from basic linear algebra.
\end{proof}

\noindent Due to the convergence of the series of matrices in \eqref{eq-convergence} and by Lemma \ref{lemma-matrixpower}, we know that the following limits exist: 
\begin{equation}
\label{eq-limit=0}
    \lim_{m\to \infty} b_{i,a_m}^{J,m}=\lim_{m\to \infty}[(B_K^{\widehat{s}})^m]_{i,a_m}=0,
\end{equation}
and 
\begin{equation}
\label{eq-limitofb}
\lim_{m\to \infty}b_{i,s}^{J,1}+b_{i,s}^{J,2}+\cdots+b_{i,s}^{J,m}=\sum_{m=1}^{\infty} [(B_K^{\widehat{s}})^m]_{i,s}=b_{i,s}^{K,\widehat{s}}.
\end{equation}

For the coefficient of $\varpi_J$ in \eqref{eq-computation}, notice that for any $J\subset \Delta$, the diagonal entries of $B_J$ are all zero. Similarly to the proof of Lemma \ref{lemma-matrixpower}, for any $l\geq 1$, we have  $$F_i^{J,l}=\sum_{j\in J}[(B_J)^{l}]_{i,j}.$$
Similar to the proof of Lemma \ref{lem: convergence of matrix series}, the series $\sum\limits_{n=0}^\infty(B_J)^n$ converges and 
$$\sum\limits_{m=0}^\infty(B_J)^m=(E-B_J)^{-1}=(\frac{1}{2}C_J)^{-1}=2C_J^{-1}.$$
So we have 
\begin{equation}\label{eq-limitofF}
    \begin{aligned}
\lim_{m\to \infty}(F_i^{J,0}+\cdots+F_i^{J,m-1})t&=\left(\sum_{m=0}^\infty\sum_{j\in J}[(B_J)^{m}]_{i,s}\right) t\\
&=t\sum_{j\in J}\left[\sum_{m=0}^\infty(B_J)^{m}\right]_{i,j}\\
&=2t\sum_{j\in J}[C_J^{-1}]_{i,j},
\end{aligned}
\end{equation}
where we can change the order of summation because for any $J\subset \Delta, B_J$ is a nonnegative matrix and the series above is a positive-term series. 

Compare \eqref{eq-limit=0}, \eqref{eq-limitofb}, \eqref{eq-limitofF} with \eqref{eq-def of matrix} and \eqref{eq-computation}, Theorem \ref{thm-main} is proved.

\end{proof}

Finally, we prove the following lemma. Combined with Corollary \ref{cor-relation} and Theorem \ref{thm-main}, it proves our main Theorem \ref{thm-intromain} in the introduction.

\begin{lemma}\label{lem: manifestly positive formula}
Let $J,K\subset \Delta$, $i\in J$, $K=J \bigsqcup \{s\}$. Then we have 
    \[
        b^{K, \widehat{s}}_{i, s}  = \frac{[C_{K}^{-1}]_{i,s}}{[C_{K}^{-1}]_{s,s}},
    \]
    where $b^{K, \widehat{s}}_{i,s}$ is defined in \eqref{eq-def of matrix},  $[C_{K}^{-1}]_{i,s}$ and $[C_{K}^{-1}]_{s,s}$  denote respectively the $(i, s)$-entry and $(s, s)$-entry of the matrix $C_{K}^{-1}$.
\end{lemma}

\begin{proof}
    We start by expressing $b^{K, \widehat{s}}_{i, s}$ as a dot product:
    \[
    b^{K, \widehat{s}}_{i, s} = \left(\text{row indexed by $i$ of } B_K^{\widehat{s}}\right) \cdot \left(\text{column indexed by $s$ of } (E - B_K^{\widehat{s}})^{-1}\right).
    \]
    We calculate each part of the dot product separately.

    Let \(e_i\) be the standard basis column vector indexed by $i$, \(x_s\) be the column vector  of \(C_K\) indexed by $s$, \((x'_s)^{\text{T}}\) be the row vector  of \(C_K\) indexed by $s$, \(y_s\) be the column vector  of \(C_K^{-1}\) indexed by $s$, and \((y'_s)^{\text{T}}\) be the row vector  of \(C_K^{-1}\) indexed by $s$.  Recall that $B_K^{\widehat{s}}$ is defined as $B_K:=\frac{1}{2}A_K$ with the entries in the row indexed by $s$ being zeros and $A_K:=2E-C_K$. Since \(i \neq s\), we have
    \begin{align*}
        \text{the row indexed by $i$ of } B_K^{\widehat{s}} 
        &= e_i^{\text{T}} B_K^{\widehat{s}} \\
        &= e_i^{\text{T}} B_K \\
        &= e_i^{\text{T}} - \frac{1}{2} (x'_i)^{\text{T}}.
    \end{align*}
    Next, consider \((E - B_K^{\widehat{s}})^{-1}\) and we compute
    \begin{align*}
        (E - B_K^{\widehat{s}})^{-1} 
        &= \left(\frac{1}{2}C_K - e_s \left(\frac{1}{2}x_s - e_s\right)^{\text{T}}\right)^{-1} \\
        &= \left(\frac{1}{2}C_K\right)^{-1} + \frac{\left(\frac{1}{2}C_K\right)^{-1} e_s \left(\frac{1}{2}x_s - e_s\right)^{\text{T}} \left(\frac{1}{2}C_K\right)^{-1}}{1 - \left(\frac{1}{2}x_s - e_s\right)^{\text{T}} \left(\frac{1}{2}C_K\right)^{-1} e_s},
    \end{align*}
    where the second equality is by the Sherman--Morrison formula \cite{sherman1950adjustment}. 
    Note that
    \begin{align*}
        \left(\frac{1}{2}x_s - e_s\right)^{\text{T}}\left(\frac{1}{2}C_K\right)^{-1}e_s &= x_s^{\text{T}}C_K^{-1}e_s - 2e_s^{\text{T}}C_K^{-1}e_s \\
        &= 1 - 2[C_K^{-1}]_{s,s}, \\
        \left(\frac{1}{2}C_K\right)^{-1} e_s \left(\frac{1}{2}x_s - e_s\right)^{\text{T}} \left(\frac{1}{2}C_K\right)^{-1} &= 2C_K^{-1}e_sx_s^{\text{T}}C_K^{-1} - 4C_K^{-1}e_se_s^{\text{T}}C_K^{-1} \\
        &= 2y_se_s^{\text{T}} - 4y_s(y'_s)^{\text{T}}.
    \end{align*}
    Then we have 
    \[
    (E - B_K^{\widehat{s}})^{-1} = 2C_K^{-1} + \frac{y_se_s^{\text{T}} - 2y_s(y'_s)^{\text{T}}}{[C_K^{-1}]_{s,s}}.
    \]
    Now we consider \(\text{the column indexed by $s$ of } (E - B_K^{\widehat{s}})^{-1}\):
    \begin{align*}
        \text{the column indexed by $s$ of } (E - B_K^{\widehat{s}})^{-1} 
        &= (E - B_K^{\widehat{s}})^{-1} e_s \\
        &= 2 C_K^{-1} e_s + \frac{y_s e_s^{\text{T}} e_s - 2 y_s (y'_s)^{\text{T}} e_s}{[C_K^{-1}]_{s,s}}.
    \end{align*}
    Therefore, we have
   \begin{align*}
        &\left[ B_K^{\widehat{s}}(E - B_K^{\widehat{s}})^{-1} \right]_{i,s} \\
        =& \left( e_i^{\text{T}} - \frac{1}{2} (x'_i)^{\text{T}} \right) \cdot \left( 2 C_K^{-1} e_s + \frac{y_s e_s^{\text{T}} e_s - 2 y_s (y'_s)^{\text{T}} e_s}{[C_K^{-1}]_{s,s}} \right) \\
        =& 2 e_i^{\text{T}} C_K^{-1} e_s - (x'_i)^{\text{T}} C_K^{-1} e_s \\
        &+ \frac{e_i^{\text{T}} y_s e_s^{\text{T}} e_s - 2 e_i^{\text{T}} y_s (y'_s)^{\text{T}} e_s - \frac{1}{2} (x'_i)^{\text{T}} y_s e_s^{\text{T}} e_s + (x'_i)^{\text{T}} y_s (y'_s)^{\text{T}} e_s}{[C_K^{-1}]_{s,s}} \\
        =& 2[C_K^{-1}]_{i,s} - \delta_{is} + \frac{[C_K^{-1}]_{i,s}-2[C_K^{-1}]_{i,s}[C_K^{-1}]_{s,s} - \frac{1}{2}\delta_{is} + [C_K^{-1}]_{s,s}\delta_{is}}{[C_K^{-1}]_{s,s}} \\
        =& \frac{[C_K^{-1}]_{i,s}}{[C_K^{-1}]_{s,s}},
    \end{align*}
    where the Kronecker delta $\delta_{is}=0$ since $i \neq s$.
\end{proof}

\begin{remark}
    Lusztig--Tits \cite{lusztig1992inverse} gave an explicit formula for the $(i, j)$-entry of a Cartan matrix $C_K^{-1}$. If $i$ and $j$ are in different connected components of $K$, then $[C_K^{-1}]_{i,j}=0$. Now assume that $i$ and $j$ are in the same connected component of $K$. Let $i_1, i_2, \ldots, i_p$ be the unique sequence in $K$ such that $i_1=i, i_p=j$ and any two consecutive terms of the corresponding simple reflections $s_{i_1}, s_{i_2}, \ldots s_{i_p}$ are non-commuting. Let $W(i, j)$ be the parabolic subgroup of $W_K$ generated by $\left\{\sigma_k ; k \in K-\left\{i_1, i_2, \ldots, i_p\right\}\right\}$ and let $C(i, j)$ be the corresponding Cartan matrix. Then we have $[C_K^{-1}]_{i,j}=\frac{\det(C(i, j))}{det(C_K)}$.
\end{remark}

Now we give some examples of computations using our formula.

\begin{eg}
    For Lie type $B_3$, the structure constants matrix of $I=\{2\}$ is
    
    \[\left(d_{I, J}^K\right)_{J, K \subset[3]}=\begin{array}{cc}
    &J\\
    \begin{pmatrix}
  0  & 0  & 1 &  0 &  0  & 0  & 0  & 0\\
  0  & 0  & 0 &  0 &  1  & 0  & 0  & 0\\
  0  & 0  & 2t &  0 & \frac{1}{2} &  0   &1  & 0\\
  0  & 0  & 0 &  0 &  0 &  0 &  1 &  0\\
  0  & 0  & 0 &  0 &  2t  & 0 &  0& \frac{4}{3}\\
  0  & 0  & 0 &  0 &  0  & 0   &0  & 1\\
  0  & 0  & 0 &  0 &  0  & 0  & 2t & 1\\
  0  & 0  & 0 &  0 &  0  & 0 &  0 &  2t
    \end{pmatrix}
    &
    \begin{array}{c}
         \emptyset\\ \{1\} \\\{2\} \\\{3\}\\\{1,2\}\\\{1,3\}\\\{2,3\}\\\{1,2,3\}     
    \end{array}
    \end{array}
    \]
    From the matrix we can see $d_{\{2\},\{1,2\}}^{\{1,2\}}=2t$, $d_{\{2\},\{1,2\}}^{\{1,2,3\}}=\frac{4}{3}$, which means $\varpi_2\cdot \varpi_{\{1, 2\}}=2t\varpi_{\{1, 2\}}+\frac{4}{3}\varpi_{\{1, 2, 3\}}.$ This matches the computation in \cite[Example 6.8]{goldin2021equivariant}
\end{eg}

It can be seen from the above that the structure constants matrix is quite sparse, hence our formula is very fast when doing actual computations. We give an example with larger rank in the non-equivariant case below.

\begin{eg}
    For Lie Type $A_9$, $I=\{3,6,8\}$, $J=\{1,3,5,6,7\}$, from our formula and letting $t=0$, one can deduce 
    \[
    \varpi_I \cdot  \varpi_J=\frac{18}{35}\varpi_{\{1,2,3,4,5,6,7,8\}}+\frac{1}{5}\varpi_{\{1,2,3,5,6,7,8,9\}}+\frac{2}{7}\varpi_{\{1,3,4,5,6,7,8,9\}}.
    \]
    Hence, by Theorem \ref{thm-PSbasis}, we have
    \begin{align*}
    p_I\cdot p_J=& \frac{1}{1!\cdot 1!\cdot1!}\cdot\frac{1}{1!\cdot1!\cdot3!}\varpi_I\varpi_J\\ 
    =&\frac{1}{3!}\left(\frac{18}{35}\varpi_{\{1,2,3,4,5,6,7,8\}}+\frac{1}{5}\varpi_{\{1,2,3,5,6,7,8,9\}}+\frac{2}{7}\varpi_{\{1,3,4,5,6,7,8,9\}}\right)\\
    =&\frac{8!}{3!}\frac{18}{35}p_{\{1,2,3,4,5,6,7,8\}}+\frac{3!\cdot5!}{3!}\frac{1}{5}p_{\{1,2,3,5,6,7,8,9\}}+\frac{7!}{3!}\frac{2}{7}p_{\{1,3,4,5,6,7,8,9\}}\\
    =&3456p_{\{1,2,3,4,5,6,7,8\}}+24p_{\{1,2,3,5,6,7,8,9\}}+240p_{\{1,3,4,5,6,7,8,9\}},
    \end{align*}
    which matches the computation for the non-equivariant case in \cite[Example 5.2]{abe2024geometry}.
\end{eg}

We have the following corollary. In particular, it means that our formulas provide an algebraic proof of the positivity of the equivariant structure constants of the Peterson Schubert calculus.

\begin{corollary}\label{cor-positivity}
The structure constant  $d_{I,J}^K$ (and the structure constant $c_{I,J}^K$) is a polynomial in $t$ with non-negative coefficients for all $I, J, K\subset\Delta$.
\end{corollary}

\begin{proof}
    It is well known that the entries of the inverse of any Cartan matrix are all non-negative (see, for example, \cite[Exercise 13.8]{humphreys1972}), hence by \eqref{eq-def of matrix} and Lemma \ref{lem: manifestly positive formula}, all matrices $(d^{K}_{i, J})_{J, K}$ in Theorem \ref{thm-main} are polynomials in $t$ with non-negative coefficients.
Since the entries of the resulting matrix product are still polynomials in $t$ with non-negative coefficients, this corollary follows directly. 
\end{proof}

We have the following corollary, which gives a simple criterion for when the structure constants are non-zero.

\begin{corollary} \label{cor-nonzero}
    The structure constant $c_{I,J}^K\neq 0$ (equivalently, $d_{I,J}^K\neq 0$) if and only if 
    \item[$\quad\bullet$] $K \supset I \cup J $, and
    \item[$\quad\bullet$] For each connected component $K_k$ of $K$, we have $|K_k|\leq|K_k\cap I| + |K_k\cap J|$.
\end{corollary}

\begin{proof}
    By Corollary \ref{cor-relation}, the structure constant $c_{I,J}^K\neq 0$ is equivalent to the structure constant $d_{I,J}^K\neq 0$.
    
    When $|I|=1$, we have $I=\{i\}$, and we consider the value of $d_{i,J}^K$. Then by \eqref{eq-def of matrix}, $d_{i,J}^K\neq 0$ if and only if we have $K=J\bigsqcup\{i\}$, or $i\in J$, $|K|=|J|+1$, $K\setminus J=\{s\}\neq\{i\}$ and $b_{i,s}^{K,\widehat{s}}\neq 0$, or $i\in J$, $K=J$ and $\sum_{k\in K}[C_K^{-1}]_{i,k}\neq 0$. In any case, we have $K \supset \{i\}\cup J$, so in necessity, it remains to prove $|K_k|\leq|K_k\cap \{i\}| + |K_k\cap J|$ for each connected component $K_k$ of $K$ in each case.
    
    For the case $K=J\bigsqcup\{i\}$, since $i$ can only be contained in exactly one connected component of $K$, it is easy to see that $|K_k|=|K_k\cap \{i\}| + |K_k\cap J|$ for each connected component $K_k$ of $K$. 
    
    Now consider the case $i\in J$, $|K|=|J|+1$, $K\setminus J=\{s\}\neq\{i\}$ and $b_{i,s}^{K,\widehat{s}}\neq 0$. In this case, it is not hard to see that $|K_k|=|K_k\cap \{i\}| + |K_k\cap J|$ for each connected component $K_k$ is equivalent to $i$ and $s$ being in the same component. By \cite{lusztig1992inverse}, the latter is equivalent to $[C_{K}^{-1}]_{i,s}\neq 0$. By Lemma \ref{lem: manifestly positive formula}, this case is proved.

    Then consider the case $i\in J$, $K=J$ and $\sum_{k\in K}[C_K^{-1}]_{i,k}\neq 0$. Note that $\sum_{k\in K}[C_K^{-1}]_{i,k}\neq 0$ always holds, since all entries of $C_K^{-1}$ are non-negative. It is easy to see that $i\in J$ and $K=J$ imply $|K_k|\le|K_k\cap \{i\}| + |K_k\cap J|$ for each connected component $K_k$.

    As for the sufficiency, we have $K \supset \{i\}\cup J$ and for each connected component $K_k$ of $K$, $|K_k|\leq|K_k\cap \{i\}| + |K_k\cap J|$. If $|K_k|=|K_k\cap \{i\}| + |K_k\cap J|$, from the discussion above, we have $K\setminus J=\{s\}$, $s=i$ or $s$ and $i$ are in the same connected component, which means $d_{i,J}^K\neq 0$. If $|K_k|<|K_k\cap \{i\}| + |K_k\cap J|$ for some $K_k$, by $K\supset \{i\} \cup J$, we have $K=J$, which also means $d_{i,J}^K\neq 0$.

    For the general $I$, choose any \( i \in I \) and suppose \( I=I' \bigsqcup \{i\}  \). By Theorem \ref{thm-main}, we have  
\[ 
d_{I,J}^K = \sum_{\substack{K'\subset\Delta}} d_{I',J}^{K'} d_{i,K'}^K. 
\]  
    Since all the terms in the summation are polynomials in $t$ with nonnegative coefficients by Corollary \ref{cor-positivity}, \( d_{I,J}^K \neq 0 \) if and only if there is a \(K'\subset\Delta\) such that \( d_{I',J}^{K'}\neq 0\) and \(d_{i,K'}^{K} \neq 0 \). We need to prove that they are further equivalent to $K \supset I \cup J $ and $|K_k|\le|K_k\cap I| + |K_k\cap J|$ for each connected component $K_k$ of $K$. 

We prove the necessity by induction on the cardinality of \( I \). So suppose \( d_{I',J}^{K'}\neq 0\) and \(d_{i,K'}^{K} \neq 0 \) for some $K'$. Because \(d_{i,K'}^{K} \neq 0 \), we have $i\in K'$ and $K'=K$, or \( K\supset\{i\}\cup K' \) and \( |K'| = |K| - 1 \). 

For the case of $i\in K'$ and $K'=K$, according to the induction hypothesis for $d_{I',J}^{K'}\neq 0$, we have
\begin{equation*}
\begin{aligned}
    |K_k|&=|K_k'| \\
    &\le |K_k' \cap I'| + |K_k' \cap J|\\
    &=|K_k \cap I'| + |K_k \cap J|\\
     &\le |K_k \cap I| + |K_k \cap J|
    \end{aligned}
\end{equation*} 
for each connected component $K_k$ of $K$, 
which completes the proof of this case.

For the case of \( K\supset\{i\}\cup K' \) and \( |K'| = |K| - 1 \), according to the induction hypothesis for $d_{I',J}^{K'}\neq 0$, \(K' \supset I '\cup J \) and for each connected component \( K_k' \) of \( K' \),  we have
\begin{equation}\label{eq-K' connected}
    |K_k'| \le |K_k' \cap I'| + |K_k' \cap J|.
\end{equation} 
It follows that \(K \supset I \cup J \).
Let \( K = K_1 \bigsqcup \dots \bigsqcup K_t \) be the decomposition of \( K \) into connected components. Without loss of generality, we assume that \( i \in K_1 \) and \( K \setminus K'=\{s\} \) (here $s$ may or may not equal $i$).  Note that \( K' \supset I '\cup J \), we have \( s \notin I' \) and \( s \notin J \). Since \(d_{i,K'}^{K} \neq 0 \), by the proof of the \(|I|=1\) case, it is not hard to see \(s \in K_1 \).  

    Assume that \( K_1 \setminus \{s\} = K_{1,1} \bigsqcup \dots \bigsqcup K_{1,r} \) is the decomposition of \( K_1 \setminus \{s\} \) into connected components. Here \( r = 0 \) means \( K_1 = \{i\}=\{s\} \), and \( r = 1 \) means \( K_1 \setminus \{s\} \) is still connected.  

    By definition, the decomposition of \( K' \) into connected components is:  
\[ 
K' = K_{1,1} \bigsqcup \dots \bigsqcup K_{1,r} \bigsqcup K_2 \bigsqcup \dots \bigsqcup K_t. 
\]  
    Define \(K_1':=K_1\setminus\{s\}= K_{1,1} \bigsqcup \dots \bigsqcup K_{1,r}.\)

    Since \( i \notin K_2, \dots, K_t \), when \( 2 \leq k \leq t \), we have \( K_k = K_k' \), the connected component of \( K' \). By \eqref{eq-K' connected}, we have \( |K_k| \le |K_k \cap I| + |K_k \cap J|, \) for $k=2, ..., t$.
 
    As for \( K_1 \) (the component containing \( i \)), we have
\begin{align*}
|K_1 \cap I| + |K_1 \cap J|=&|K_1 \cap (I' \bigsqcup \{i\})| + |K_1 \cap J| \\
=&|K_1 \cap (I' \bigsqcup \{i\})| + |K_1' \cap J| \, \\
=&|K_1 \cap I'| + 1 + |K_1' \cap J| \\
=&|K_1' \cap I'| + 1 + |K_1' \cap J| \\
\ge&|K_1'|+1\\
=&|K_1|,
\end{align*}
where the second equality is because $s\notin J$, the fourth equality is because $s\notin I'$, the fifth inequality is by \eqref{eq-K' connected}. The necessity is proved.

    As for the sufficiency, assume that \( K \supset I \cup J \), and for each connected component \( K_k \) of \( K \), we have   
\[ 
|K_k| \le |K_k \cap I| + |K_k \cap J|. 
\]  
Recall that \( I = \{i\} \bigsqcup I' \), and now we define \( K' = K \setminus \{i\} \). By \eqref{eq-def of matrix}, we have \(d_{i,K'}^K = 1 \).
    By definition, \( K' \supseteq I' \cup J \), so we have $i\notin J$. Using the same arguments as above, we know that each connected component \( K_k' \) of \( K' \) satisfies 
\[ 
|K_k'| \le |K_k' \cap I'| + |K_k' \cap J|. 
\]  
Applying the induction hypothesis to $I'$, we have \( d_{I',J}^{K'} \neq 0 .\) 
    This finishes the proof of the sufficiency. 
\end{proof}

\begin{remark} \label{rmk-nonzero}
By tracing the above proof carefully, one can see that the non-equivariant structure constant $m_{I,J}^K\neq 0$ if and only if $K \supset I \cup J $ and
$|K_k|=|K_k\cap I| + |K_k\cap J|$ for each connected component $K_k$ of $K$. 
\end{remark}

\begin{remark}
    Goldin--Gorbutt proved Corollary \ref{cor-nonzero} in the case of Lie type $A$ in \cite[Theorem 8]{goldin2022positive}.
\end{remark}

\section{Applications to mixed Eulerian numbers} \label{sec-Applications}
Firstly, we introduce the mixed $\Phi$-Eulerian numbers for arbitrary Lie types as in \cite{postnikov2009permutohedra} and their connections to the structure constants of the Peterson Schubert calculus. We mainly follow the notation as in \cite{Horiguchi2024mixed}.

Recall that \(\Phi\) is a crystallographic root system of rank \(n\). Let \(\Lambda\) be the associated integral weight lattice and \(\Lambda_{\mathbb{R}} = \Lambda \otimes \mathbb{R}\) be the weight space. The associated Weyl group \(W\) acts on the weight space \(\Lambda_{\mathbb{R}}\)  as a finite real reflection group. Taking a weight \(\chi \in \Lambda_{\mathbb{R}}\), the \emph{weight polytope} \(P_\Phi(\chi)\) is defined as the convex hull of the Weyl group orbit of \(\chi\):

\[P_\Phi(\chi) := \mathrm{ConvexHull}\{w(\chi) \in \Lambda_{\mathbb{R}} \mid w \in W\}.\]

\begin{eg}[Type $A_{n-1}$ weight polytope---Permutohedron] 
Let $n$ be a positive integer. The permutation group $S_n$ in $n$ letters acts naturally on $\mathbb{R}^n$ by permuting coordinates. For $a_1, \ldots, a_n \in \mathbb{R}$, the \emph{permutohedron} $P_n\left(a_1, \ldots, a_n\right)$ is defined to be the convex hull of points in the $S_n$-orbit of  $(a_1, \ldots, a_n)$
\[
P_n\left(a_1, \ldots, a_n\right):= \mathrm{ConvexHull}\left\{\left(a_{w(1)}, \ldots, a_{w(n)}\right) \in \mathbb{R}^n \mid w \in S_n\right\}.
\]
It is at most $(n-1)$-dimensional, sitting inside an affine hyperplane in $\mathbb{R}^n$. If $a_1, a_2, \ldots, a_n$ are distinct, then $P_n\left(a_1, \ldots, a_n\right)$ is $(n-1)$-dimensional, and Postnikov gave a formula \cite[Theorem 3.1]{postnikov2009permutohedra} of its volume using \emph{divided symmetrization}. This formula can be interpreted as a computation in the equivariant cohomology of the type $A$ permutohedral variety using the localization technique; see \cite[Section 3]{Horiguchi2024mixed}.
\end{eg}

Recall that \(\Delta := (\alpha_1, \ldots, \alpha_n)\) denotes the set of simple roots in \(\Phi\). In \cite{postnikov2009permutohedra}, Postnikov gave a formula for the volume of the weight polytope \(P_\Phi(\chi)\), with the normalization that the volume of the parallelepiped generated by the simple roots \(\alpha_1, \ldots, \alpha_n\) is 1.

Let \(\varpi_1, \ldots, \varpi_n\) be the fundamental weights. Suppose \(\chi = u_1 \varpi_1 + \cdots + u_n \varpi_n\) and consider the associated weight polytope \(P_\Phi(\chi)\). Its volume is a homogeneous polynomial \(V_\Phi\) of degree \(n\) in variables \(u_1, \ldots, u_n\):
\begin{equation}\label{eq-volume}
   V_\Phi(u_1, \ldots, u_n) := \text{volume of } P_\Phi(u_1 \varpi_1 + \cdots + u_n \varpi_n). 
\end{equation}

Postnikov \cite{postnikov2009permutohedra} defined the \emph{mixed \(\Phi\)-Eulerian numbers} \(A_{c_1, \ldots, c_n}^\Phi\), for natural numbers \(c_1, \ldots, c_n \) with \(c_1 + \cdots + c_n = n\), as the coefficients of the volume polynomial \eqref{eq-volume}:
\begin{equation} \label{eq-def of MEN}
    V_\Phi(u_1, \ldots, u_n) = \sum_{c_1, \ldots, c_n} A_{c_1, \ldots, c_n}^\Phi \frac{u_1^{c_1}}{c_1!} \cdots \frac{u_n^{c_n}}{c_n!}.
\end{equation}
By this definition, the mixed $\Phi$-Eulerian number $A_{c_1,\ldots,c_n}^{\Phi}$ is exactly the \emph{mixed volume} of $c_1$ copies of $P_{\Phi}\left(\varpi_1\right)$, $c_2$ copies of $P_{\Phi}\left(\varpi_2\right), \ldots$, and $c_n$ copies of $P_{\Phi}\left(\varpi_n\right)$, multiplied by $n!$. Here, the weight polytopes $P_{\Phi}\left(\varpi_1\right), \ldots, P_{\Phi}\left(\varpi_n\right)$ are called the \emph{$\Phi$-hypersimplices}. The mixed $\Phi$-Eulerian numbers are known to be non-negative integers; see \cite{postnikov2009permutohedra} for more details. When $\Phi$ is of type $A$, these numbers are simply called the \emph{mixed Eulerian numbers}.

Postnikov provided in \cite{postnikov2009permutohedra} a combinatorial formula for the mixed $\Phi$-Eulerian numbers in terms of certain binary trees. In \cite{Berget2023log}, Berget--Spink--Tseng studied the log-concavity of matroid $h$-vectors in relation to the mixed Eulerian numbers, using the fact that the cohomology ring of the type $A$ Permutohedral variety is exactly the Chow ring of the Boolean matroid. In \cite{Nadeau2023permutahedral}, Nadeau--Tawari found a beautiful relation between the mixed $\Phi$-Eulerian numbers and intersection numbers of Schubert varieties and the permutohedral variety for arbitrary Lie types. In \cite{Horiguchi2024mixed}, Horiguchi showed that the mixed Eulerian numbers can be written as intersection numbers of Schubert divisors in the Peterson variety for an arbitrary Lie type as follows.

\begin{theorem}[\protect{\cite[Theorem 1.1]{Horiguchi2024mixed}}]\label{thm-M=I}
    Let \(\Phi\) be an irreducible root system of rank $n$. Let \(c_1, \ldots, c_n\) be non-negative integers with \(c_1 + \cdots + c_n = n\). 
    Then the mixed \(\Phi\)-Eulerian number \(A_{c_1, \ldots, c_n}^{\Phi}\) is equal to\[A_{c_1, \ldots, c_n}^{\Phi} = \int_{\Pet_{G}} \varpi_1^{c_1} \varpi_2^{c_2} \cdots \varpi_n^{c_n},\]
    where, as before, $\Pet_{G}$ denotes the Peterson variety associated with the simple algebraic group $G$, $\varpi_i\in H^{2}(\Pet_{G};\mathbb{Q})$ denotes the first Chern class of the line bundle $L_{\varpi_{i}}$ on $\Pet_{G}$, which is also the image of the Schubert class $\sigma_{s_i} \in H^{2}(G/B;\mathbb{Q})$ under the restriction map $ H^{2}(G/B;\mathbb{Q}) \rightarrow H^{2}(\Pet_{G};\mathbb{Q})$.
\end{theorem}

\begin{remark}
From \eqref{eq-def of MEN} and the above theorem, it follows that the volume polynomial in \eqref{eq-volume} has the following expression
\begin{equation}
    \operatorname{Vol} P_{\Phi}\left(u_1 \varpi_1+\cdots+u_n \varpi_n\right)=\frac{1}{n!} \int_{\Pet_{G}}\left(u_1 \varpi_1+\cdots+u_n \varpi_n\right)^n,  
\end{equation}
which is the volume polynomials of the nef divisors $\varpi_1, \ldots, \varpi_n$ on the Peterson variety $\Pet_{G}$. While the Peterson variety is singular in general \cite[Theorem 6]{kostant1996flag}, its rational cohomology ring $H^{*}(\Pet_{G};\mathbb{Q})$ in \eqref{eq-HHMpre} admits the structure of the cohomology ring of a (rational smooth) toric orbifold (see, for example, \cite{abe2023peterson} and \cite{gui2024weyl}), which satisfies all the K\"ahler package---the Poincar\'e duality, the hard Lefschetz theorem, and the Hodge--Riemann bilinear relation. This gives a different proof that $\operatorname{Vol} P_{\Phi}\left(u_1 \varpi_1+\cdots+u_n \varpi_n\right)$ is \emph{Lorentzian} in the sense of \cite{branden2020lorentzian}.
\end{remark}

Using the above theorem, Horiguchi gave a combinatorial model introduced in \cite{abe2024geometry} for the computation of the mixed Eulerian numbers, and derived a type-by-type computation for the mixed $\Phi$-Eulerian numbers for general Lie types by iteratively applying the Monk formula of Drellich in \cite{drellich2015monk}. As an application of our main theorem, we now derive a type-uniform formula for the mixed $\Phi$-Eulerian numbers in arbitrary Lie types. 

\begin{theorem}\label{thm-MEN}
    Let \(\Phi\) be an irreducible root system. Let \(c_1, \ldots, c_n\) be non-negative integers with \(c_1 + \cdots + c_n = n\).  
The mixed ${\Phi}$-Eulerian number $A_{c_{1}, \ldots, c_{n}}^{\Phi}$ can be computed using the following formula:
\[
A_{c_{1},\ldots,c_{n}}^{\Phi} = \frac{|W_{\Phi}|}{\det(C_\Phi)} 
\left[ M_{1}^{c_{1}} \cdots M_{n} ^{c_{n}} \right]_{\emptyset, \Delta}.
\]
where $M_{i}$ is the matrix defined in \eqref{eq-intro-Nmatrix}.
The notation 
\(
[\quad]_{\emptyset, \Delta}
\)
denotes the entry in the row indexed by $\emptyset$ and the column indexed by $\Delta$ of a matrix.
\end{theorem}

\begin{proof}
According to Theorem \ref{thm-M=I}, we have:
\[
A_{c_{1},\ldots ,c_{n}}^{\Phi} = \int_{\Pet_{G}} \varpi_{1}^{c_{1}} \cdots \varpi_{n}^{c_{n}}.
\]
Using Theorem \ref{thm-introNmain} and Corollary \ref{cor-relation}, it follows that:
\[
\varpi_{1}^{c_{1}} \cdots \varpi_{n}^{c_{n}} = a_{\Delta} \, \varpi_{1} \cdots \varpi_{n},
\]
where 
\[
a_{\Delta}=\left[ \left( (m_{1,J}^{K})_{J,K} \right)^{c_{1}} \cdots \left( (m_{n,J}^{K})_{J,K} \right)^{c_{n}} \right]_{(\emptyset, \Delta)}.
\]
By \cite[Theorem 1.1.(3)]{Horiguchi2024mixed}, we conclude:
\[
A_{c_{1},\ldots, c_{n}}^{\Phi} = a_{\Delta} \int_{\Pet_{G}} \varpi_{1} \cdots \varpi_{n} = \frac{|W_{\Phi}|}{\det(C_\Phi)} a_{\Delta}.
\]
\end{proof}

The table of explicit values of $\frac{|W_{\Phi}|}{\det(C_{\Phi})}$ is presented below. Now we give some examples using our formula to compute the mixed $\Phi$-Eulerian numbers.

\begin{table}[htb]
  \begin{tabular}{|c|c|} \hline
    $\Phi$ & $\frac{|W|}{\det (C_{\Phi})}$ \\ \hline \hline
    $A_n$ & $n!$ \\ \hline
    $B_n,C_n$ & $2^{n-1} \cdot n!$ \\ \hline
    $D_n$ & $2^{n-3} \cdot n!$ \\ \hline
    $E_6$ & $2^{7} \cdot 3^{3} \cdot 5$ \\ \hline
    $E_7$ & $2^{9} \cdot 3^{4} \cdot 5 \cdot 7$ \\ \hline
    $E_8$ & $2^{14} \cdot 3^{5} \cdot 5^{2} \cdot 7$ \\ \hline
    $F_4$ & $2^{7} \cdot 3^{2}$ \\ \hline
    $G_2$ &  $2^{2} \cdot 3$ \\ \hline
  \end{tabular}
  \caption{A list of values of $\frac{|W_{\Phi}|}{\det (C_{\Phi})}$.}
\label{tab:A list of values mPhi}
\end{table}

\begin{eg}
    For Lie type $A_8$ and $(c_1,\ldots,c_8)=(1,0,2,3,0,0,1,1)$. Then from our formula, it is easy to compute
    \begin{align*}A_{c_1,\ldots,c_8}^{\Phi}&=8!\left[M_{1}M_{3}^2M_{4}^3M_{7}M_{8}\right]_{\emptyset,\{1,2,3,4,5,6,7,8\}}\\ 
    &=8!\cdot\frac{41}{70}\\
    &=23616,
    \end{align*}
    which matches the computation in \cite[Example 5.4]{Horiguchi2024mixed}.
\end{eg}

\begin{eg}
    For Lie type $E_6$ and $(c_1,\ldots,c_6)=(0,1,0,2,3,0)$. Then from our formula, it is easy to compute
    \begin{align*}
        A_{c_1,\ldots,c_6}^{\Phi}&=\frac{|W_{E_6}|}{\det(C_{E_6})}\left[M_{2}M_{4}^2M_{5}^3\right]_{\emptyset,\{1,2,3,4,5,6\}}\\
        &= 2^7\cdot 3^3\cdot5\cdot\frac{81}{40}\\
        &=34992,
    \end{align*}
    which matches the computation in \cite[Example 7.21]{Horiguchi2024mixed}. 
\end{eg}

It can be seen from the above examples that our formula avoids the need to discuss the changes in the Lie types case-by-case when considering sub-root systems as in \cite[Section 7]{Horiguchi2024mixed}. The only thing needed is to substitute $(c_1,\ldots,c_n)$ into our formula to obtain the results.

\bibliographystyle{amsplain}
\bibliography{template}

\end{document}